\documentclass[11pt]{amsproc}
\usepackage{graphicx, amssymb,amsmath, latexsym,cite}
\usepackage{amsthm, xcolor}
\usepackage{lscape}  
\usepackage{float}
\usepackage{fancyhdr, lastpage} 
\usepackage{times,url}    
\usepackage{color,multicol}

\newtheoremstyle{mythm}                   
{6pt}
{6pt}
{\it}
{}
{\bf}
{.}
{.5em}
{}

\newtheoremstyle{mydef}                   
{6pt}
{6pt}
{}
{}
{\bf}
{.}
{.5em}
{}

\newtheoremstyle{myrem}                   
{6pt}
{6pt}
{}
{}
{\bf}
{.}
{.5em}
{}

\theoremstyle{mythm}       
\newtheorem{theorem}{Theorem}

\newtheorem{lemma}[theorem]{Lemma}
\newtheorem{corollary}[theorem]{Corollary}

\theoremstyle{mydef}      
\newtheorem{definition}[theorem]{Definition}

\theoremstyle{myrem}
\newtheorem{remark}[theorem]{Remark}
\numberwithin{equation}{section}

\newcommand{\M}{\mathbb{M}}
\newcommand{\B}{\mathbb{B}}
\newcommand{\BB}{{\rm B}}
\newcommand{\myA}{{\rm A}}

\newcounter{ithmcount}

\newcounter{itemscount}

\allowdisplaybreaks

\begin{document}

\vspace*{-0.5cm} 

\title{Finding $59{:}29$ in the Monster}

\begin{abstract}
The classification of the maximal subgroups of the Monster group has been completed recently, and explicit generators for each such subgroup (up to conjugacy) have been made available for the software  \texttt{mmgroup}, with the exception of the maximal subgroup $59{:}29$, see Dietrich \emph{et al}.\ (Adv.\ Math., 2025; J Algebra, 2026). We provide explicit generators for this last maximal subgroup and comment on the extensive search that led to finding them. Our method is similar to, but significantly more involved than Bray \emph{et al.}'s (London Math.\ Soc.\ J.\ Comput.\ Math., 2016)  approach for constructing $47{:}23$ in the Baby Monster. Our result allows us to provide a new short proof that the Monster does not have a subgroup ${\rm PSL}_2(59)$, correcting a result of Holmes and Wilson (J.\ London Math.\ Soc., 2004). 
\end{abstract} 
 
\author[H.~Dietrich]{Heiko Dietrich}
\author[M.\ Lee]{Melissa Lee}
\author[A.\ Pisani]{Anthony Pisani}
\author[A.\ Rizzoli]{Aluna Rizzoli}
\address[Dietrich, Lee, Pisani]{School of Mathematics, Monash University, Clayton, Australia}
\address[Rizzoli]{Department of Mathematics, King's College London, UK \\
Heilbronn Institute for Mathematical Research, Bristol, UK}
\email{\{heiko.dietrich,melissa.lee,anthony.pisani\}@monash.edu}
\email{aluna.rizzoli@kcl.ac.uk}
\date{\today}

\thanks{Dietrich, Lee, and Rizzoli thank the Isaac Newton Institute for Mathematical Sciences, Cambridge, for support and hospitality during the programme \emph{Algebraic groups, geometry, invariants and related topics}, where work on this paper was undertaken; this work was supported by EPSRC grant EP/Z000580/1. Rizzoli also acknowledges support by the Additional Funding Programme for Mathematical Sciences, delivered by EPSRC (EP/V521917/1) and the Heilbronn Institute for Mathematical Research. The authors thank Adam Thomas for helpful discussions that promoted this collaboration.}

\maketitle

\vspace*{-0.5cm}

\section{Introduction}
\noindent The Monster group $\M$, the largest of the sporadic simple groups, was constructed by Griess \cite{griess_friendly_giant} as the automorphism group of a non-associative algebra structure on a $196{,}884$-dimensional complex vector space $V$.
Decades of work were required to complete the classification of the maximal subgroups of the Monster group, and the final (computational) proofs were only made possible with the help of Seysen's software package \texttt{mmgroup} \cite{mmgroup}; we refer to \cite{dlp} for extensive background and references. This software package also allowed Dietrich \emph{et al}.\ to provide explicit generators for each of these maximal subgroups, up to conjugacy, with the exception of the maximal subgroup $P=59{:}29$, see \cite{dlpp}. The group $P$ is generated by an element $y$ of order $59$ and an element $n$ of order $29$ that acts nontrivially on $\langle y\rangle$. While it is straightforward to find elements of orders $59$ and $29$ in $\M$, the group size $|\M|\approx 8\times 10^{53}$ makes it infeasible to find  generators of $P$ by a simple random search: as explained in \cite[Remark~5.7]{dlpp}, a naive approach yields an approximately $1$ in $10^{48}$ chance of success on each attempt.

The construction of a subgroup $47{:}23$ in the Baby Monster $\B$ described in \cite{find_47_23} faced similar problems, which were overcome  using \emph{fingerprinting} and an analysis based on the birthday paradox. If a fingerprint collision was detected, an explicit basis calculation was used to construct a certain conjugating matrix, leading to generators for $47{:}23$. A direct application of this method to the construction of $P<\M$ is not possible
since an analogous explicit basis calculation would not be feasible in $\M$ and since the estimated time of success is in years,  see \cite[Remark~5.7]{dlpp}. 

Here we present a new approach that is also based on fingerprinting. The main difference from the approach in \cite{find_47_23} is how a fingerprint collision is exploited: we use it as a necessary condition for element conjugacy, 
and then attempt to construct a conjugating element by sifting through a chain of subgroups where computations become feasible. In contrast, subgroup chains are employed in \cite{find_47_23} only to write the generators of $47{:}23$ as words in the standard generators of $\mathbb{B}$. The details of our fingerprinting approach are explained in Section \ref{sec_approach}. Our main result is the following; as for the results in \cite{dlp,dlpp}, our generators are expressed  using the Python software package \texttt{mmgroup} developed by Seysen, see \cite{mmgroup, fast_monster}.

\begin{theorem}
  \label{thm:mainthm}
The elements $y,n$ defined in the \texttt{mmgroup} code provided in  Figure \ref{fig:gens} generate a maximal subgroup $59{:}29$ of the Monster group; specifically, $y$ and $n$ satisfy the presentation $\langle y,n\mid y^{59}, n^{29}, y^n=y^3\rangle$.
\end{theorem}

\begin{proof}
It has been shown in \cite{dlp,dlpp} that $\M$ has a unique conjugacy class of subgroups $59{:}29$, and that each such subgroup is maximal; see \cite[Remark 1.8]{dlp} and \cite[Section 5.6]{dlpp}. The \texttt{mmgroup} code  in Figure \ref{fig:gens} shows that the generators $y,n$ generate a subgroup  $59{:}29$ of $\M$.
\end{proof}

The explicit $59{:}29$ constructed above, together with an involution inverting $n$, can be used to prove that $\M$ has no subgroup that is isomorphic to ${\rm PSL}_2(59)$. The existence of a (maximal) subgroup ${\rm PSL}_2(59)$ had originally been claimed in \cite{psl}, and was later refuted in \cite{dlp,dlpp}, see \cite[Section 5.6]{dlpp} for more details. Theorem \ref{thm:mainthm} allows us to provide a new short proof for the non-existence of such a subgroup.

\begin{corollary}
\label{cor:no-psl}
The Monster group  has no subgroup isomorphic to ${\rm PSL}_2(59)$.
\end{corollary}

\begin{proof}
Let $t,w\in\M$ be the elements defined in
Figure~\ref{fig:gens}, so $t$ is an involution inverting $n$,
while $w$ has order $3$ and centralises both $n$ and $t$. The character table data of $\M$ implies that  $|C_\M(n)|=87$, and so  $C_\M(n)=\langle n,w\rangle\cong C_{29}\times C_3$. If $s$ is any involution inverting $n$, then $st\in C_\M(n)$, so
$st=n^i w^j$ for some $i,j$ with $0\leq i \leq 28$ and $0\leq j \leq 2$. Since $t$ inverts $n$ and centralises
$w$, the relation $s^2=1$ gives $w^{2j}=1$, so $j=0$.
Thus, the complete list of involutions inverting $n$ is
$\mathcal{I}=\{tn^j:0\leq j<29\}$. 

Now suppose, for a contradiction, that  $\M$ has a subgroup  $H\cong {\rm PSL}_2(59)$. Up to conjugation, we may assume that $\langle y\rangle$ is a Sylow $59$-subgroup of $H$. By \cite[Section 11]{odd}, the normaliser of $\langle y\rangle$ in  $H$ and in
$\M$ is $N_H(\langle y\rangle)=\langle y,n\rangle$. A direct calculation  shows that ${\rm PSL}_2(59)$ contains an involution $s$ such that $n^s=n^{-1}$ and $|ys|=3$. However, this contradicts the calculation in Figure~\ref{fig:gens} which shows that $3$ is not an order occurring in $\{|yi| : i \in \mathcal{I}\}$.
\end{proof}

Since the verification of our main results, Theorem \ref{thm:mainthm} and Corollary \ref{cor:no-psl}, is straightforward, the main purpose of this paper is  to explain  the mathematical approach that allowed us to determine  the generators of the subgroup $59{:}29$; we describe this in Section \ref{sec_hunt}, and refer to Section \ref{sec_approach} for an overview. In Section \ref{sec_comp} we provide some technical details concerning the computation.

\begin{figure}[pt]
\footnotesize  
\begin{verbatim}
from mmgroup import MM

######## computations for the proof of Theorem 1

y = MM( "M<y_64ch*x_173fh*d_6a5h*p_197619231*l_1*p_2027520*l_1"
        "*p_22311973*l_1*t_1*l_2*p_2956800*l_1*p_951578*t_2*l_2"
        "*p_2597760*l_1*p_43153362*t_1*l_1*p_1499520*l_1*p_43634321"
        "*t_1*l_2*p_2956800*l_1*p_42672208*t_2*l_2*p_2386560*l_2"
        "*p_42731883*t_1*l_2*p_1943040*l_2*p_53795081>" )

n = MM( "M<y_2a1h*x_1f01h*d_0abch*p_11763597*l_1*p_2027520*l_1"
        "*p_22746531*l_2*t_1*l_2*p_2597760*l_1*p_32474145*l_1*t_1"
        "*l_1*p_1457280*l_2*p_1487536*l_2*t_1*l_2*p_2956800*l_1"
        "*p_42677963*t_2*l_2*p_2386560*l_2*p_64121863*t_2*l_2"
        "*p_1920*l_2*p_10668720*l_2*p_5177280*t_1*l_2*p_2386560"
        "*l_2*p_86283412*t_1*l_2*p_1985280*l_1*p_42789696>" )

y.order() == 59 and n.order() == 29 and y**n == y**3
# true

######## computations for the proof of Corollary 2

t = MM( "M<y_55bh*x_806h*d_443h*p_30349552*l_2*p_2597760*l_1"
        "*p_64078716*t_1*l_1*p_1457280*l_2*p_32093344*l_1*t_2*l_1"
        "*p_1415040*l_1*p_10667712*l_2*p_2875200*t_1*l_2*p_24000*l_2"
        "*p_10689216*t_2*l_2*p_2597760*l_1*p_86280548*l_2*p_23040"
        "*t_2*l_1*p_1499520*l_2*p_43182105*l_2>" )

w = MM( "M<y_5cch*x_0a8h*d_0b37h*p_49335609*l_1*p_1858560*l_2"
        "*p_11172912*t_1*l_2*p_2787840*l_2*p_33395079*l_2*t_2*l_1"
        "*p_1393920*l_1*p_467760*l_2*p_2438400*t_1*l_2*p_2787840*l_2"
        "*p_47234*t_1*l_2*p_1943040*l_2*p_16512*l_1*t_1*l_2"
        "*p_2597760*l_1*p_64003488*t_2*l_1*p_2027520*l_1*p_18250>" )

inverters = [t * n**i for i in range(29)]
orders    = [(y * s).order() for s in inverters]

( t.order() == 2  and  n**t == n**-1  and w.order() == 3
  and  n*w == w*n  and w**t == w
  and  all(s.order() == 2 and n**s == n**-1 for s in inverters)
  and  all((t*n**i*w**j).order() == 6 for i in range(29) for j in (1, 2))
  and  3 not in orders )
# true
\end{verbatim}
\caption{Python code supporting the proofs of Theorem \ref{thm:mainthm} and Corollary \ref{cor:no-psl}. Elements $y$ and $n$ generate a subgroup $59{:}29$ in $\M$ in \texttt{mmgroup}, version 1.0.8. The elements $t$ and $w$ are used in the proof of Corollary~\ref{cor:no-psl}.} 
\label{fig:gens}
\end{figure}

\section{The hunt for generators}\label{sec_hunt} 
The subgroup  $59{:}29$ we seek is generated by an element $y$ of order $59$ and an element $n$ of order $29$ that acts nontrivially on $\langle y\rangle$.  The Monster has a unique conjugacy class of subgroups of order $59$, and $y$ can be readily found in \texttt{mmgroup} by random search in $\M$. Throughout, $y$ denotes this fixed element in $\M$ of order $59$, and our task is to construct a suitable element $n\in \M$.

\subsection{Our approach}\label{sec_approach}
The fixed element $y\in \M$ has order $59$ and the automorphism group of $\langle y\rangle$ has a unique subgroup of size $29$, generated by $y\mapsto y^3$. Thus, we seek an element $n$ of order $29$ that satisfies $y^n=y^{(3^k)}$ for some $k\in\{1,\ldots,28\}$. By replacing $n$ by a suitable power, we can then arrange that $y^n=y^3$ as in Theorem \ref{thm:mainthm}.

Our construction of $n$ is via a multi-step random search based on the birthday paradox, utilising the action on $2\myA$-axes in the $196{,}884$-dimensional representation $\varrho\colon \M\to \mathrm{GL}(V)$, and sifting the problem through the subgroup chain
\[\M> 2\cdot \B > 2^{1+23}.{\rm Co}_2 > 2^{1+23}>1.\]
We briefly provide a conceptual description; details and proofs are given in subsequent sections.

In the first step of our algorithm, we avoid working directly in $\M$ and instead consider $2\myA$-axes in the image of the representation $\varrho$, see Section \ref{sec_axes}. For any such $2\myA$-axis $t_u\in V$ we define a \emph{fingerprint} with the following property: if $g\in \M$ satisfies  $y^g=y^{(3^k)}$ for some $k\in\{1,\ldots,28\}$, then the axes $t_u$ and $t_u\cdot g$ have the same fingerprint and the exponent $k$ can be recovered, see Section \ref{sec_decidegood}. Using an extensive random search, we construct a pair of axes $t_a$ and $t_b$ that have the same fingerprint with candidate exponent $k$; the probability of success of this search is based on the birthday paradox, see Remark \ref{rem_prob1}. The next step is then to attempt the construction of $n\in \M$ such that $t_b=t_a\cdot n$ and $y^n=y^{(3^k)}$. In Section \ref{sec_redK} we describe how this problem can be reduced to deciding whether two elements $A$ and $B$ of order $59$ are conjugate by an element in the subgroup $K=2\cdot\mathbb{B}$: here $A$ and $B$ are specific conjugates of $y$ and $y^{(3^k)}$, respectively.

In Section \ref{sec_redH} we define another fingerprint for pairs comprising a $2\myA$-axis and a group element in $\M$, with the property that every $c\in K$ with $A^c=B$ defines two pairs with the same fingerprint. We use a random search to find axes $t_u,t_v$ such that  the pairs $(t_u,A)$ and $(t_v,B)$ have the same fingerprint; the probability of success is again based on the birthday paradox, see Remark \ref{rem_prob2}. We then describe how finding $c\in K$ with $A^c=B$ can  be reduced to deciding whether two elements $A'$ and $B'$ of order $59$ are conjugate by an element in $H=2^{1+23}.{\rm Co}_2$; as before, $A'$ and $B'$ are specific conjugates of $A$ and $B$, respectively.

In Section \ref{sec_redQ}, we apply a stabiliser-chain argument and reduce the conjugacy problem in $H$ to the problem of deciding whether two elements $A''$ and $B''$ of order $59$ are conjugate by an element in $Q=O_2(H)=2^{1+23}$; again, $A''$ and $B''$ are specific conjugates of the previous $A'$ and $B'$, respectively. This problem in $Q$ is then solved by a brute-force search.

We note that computations in \texttt{mmgroup} are often done by reducing the matrix coefficients of the image of $\varrho$ modulo $15$, see \cite{fast_monster}, which explains why we sometimes refer to vectors with entries in $\mathbb{Z}_{15}=\{0,\ldots,14\}$.

\subsection{$2\myA$-axes}\label{sec_axes}
The Monster has two classes of involutions, called $2\myA$ and $2\BB$, with centralisers $2\cdot \B$ and $2^{1+24}.{\rm Co}_1$, respectively. In the $196{,}884$-dimensional ordinary representation $\varrho\colon\M\to \mathrm{GL}(V)$, every $2\myA$-involution $u\in \M$ is associated with a distinguished vector $t_u\in V$ that is fixed by $u$; this vector is called the $2\myA$-axis corresponding to $u$. The map $u\mapsto t_u$ is injective, and if $g\in \M$, then $t_u\cdot g=t_{g^{-1}ug}$; in particular, the stabiliser of $t_u$ under the action of $\M$ is the centraliser $C_\M(u)\cong 2\cdot \B$. We refer to \cite[Lemma 8.4.1]{ivanov_majorana} or \cite[Section~2.1]{axes} for details.
 
Fix a $2\myA$-involution $z\in \M$ with associated axis $t_z$, so the set $X=\{t_z\cdot g:g\in\M\}$ of all $2\myA$-axes has size $|X|=|\M|/(2|\B|)$. A random $2\myA$-axis can be constructed as $t_z\cdot g$ where $g\in \M$ is chosen randomly; we note that \texttt{mmgroup} allows the construction of random $g\in \M$.
 
\begin{definition} A pair of axes $(t_u,t_v)\in X\times X$ is \emph{good} if there is $g\in\M$ and $k\in\{1,\ldots,28\}$ such that $t_u\cdot g=t_v$  and $y^g=y^{(3^k)}$; the element $g$ is a \emph{witness} to the goodness of this pair.
\end{definition}

\begin{remark}
Note that the element $n\in \M$ we are looking for has order $29$ and satisfies $y^n=y^{(3^k)}$ for some $k\in\{1,\ldots,28\}$. Hence if $t_u\in X$ is any $2\myA$-axis, then $(t_u,t_u\cdot n)$ is a good pair, and $n$ is a witness. Note that if $g\in\M$ also satisfies $t_u\cdot g=t_u\cdot n$ and $y^n=y^g$, then $u^g=u^n$, and so  $ng^{-1}$ lies in $C_\M(u)=2\cdot \B$ and in $C_\M(y)=\langle y\rangle$. Since the latter two groups have coprime orders, $ng^{-1}=1$, that is, $n=g$ is the unique witness. Our approach for finding the generator $n$ is to determine a good pair of $2\myA$-axes and a witness. 
\end{remark}

\begin{lemma}
  For each fixed $t_u\in X$ the number of good pairs $(t_u,t_v)\in X\times X$ is $28\cdot 59=1652$.
\end{lemma}
\begin{proof}
  For every  $k\in\{1,\ldots,28\}$, the set $M_k=\{g\in \M: y^g=y^{(3^k)}\}$ is a coset of $C_\M(y)=\langle y\rangle$, say  $M_k=\langle y\rangle g_0$ for $g_0\in M_k$; in particular, it has size $59$. Thus, the good pairs are   $(t_u,t_u\cdot g)$ with $g\in M_1\cup\ldots\cup M_{28}$. Suppose $g\in M_k$ and $g'\in M_{k'}$ satisfy $t_u\cdot g=t_u\cdot g'$. Then $c=g'g^{-1}\in C_\M(u)\cong 2\cdot \B$ and $y^c=y^{(3^{k'-k})}\in\langle y\rangle$, so $c\in N_\M(\langle y\rangle)$. As before, this forces $c=1$, and  the claim follows.
\end{proof}

\begin{remark}\label{rem_prob1}
It follows that a random unordered pair of sampled axes is good with probability $p=1652/|X|$. Thus, if one samples $m$ random axes, the expected number of good pairs is $\lambda_m={m\choose 2}p$. The probability $p$ is small, so for small $m$ (compared to $|X|$), the number of good pairs among $m$ randomly chosen axes can be approximated by a Poisson distribution with expectation $\lambda_m$. That is, the probability $P(m)$ that $m$ random axes contain at least one good pair is approximately $1-e^{-\lambda_m}$. In particular, $P(m)>0.5008$  for  $m=286{,}000{,}000$. We emphasise that we are not seeking a precise probability estimate, but merely a heuristic indication of the number of random axes required.
\end{remark}

The approach now is as follows: sample random axes, try to detect whether there is a good pair, and for each such good pair try to construct a witness.

\subsection{A necessary condition for goodness}\label{sec_decidegood}
We now describe necessary (but not sufficient) criteria that help us to decide whether a pair of axes is good.  For this we observe that the axes of a good pair share some symmetries.

Griess \cite[pp.\ 4,40,58,59]{griess_friendly_giant} proved that the space $V$ has an $\M$-invariant form, that is, for $u,v\in V$ and $g\in\M$ we have $\langle u\cdot g,v\cdot g\rangle =\langle u,v\rangle$. We use this form to define a \emph{fingerprint} for an axis.

Let $t_a$ be the $2\myA$-axis corresponding to the involution $a\in 2\myA$. For $k\in\{0,\ldots,28\}$ define $q_k=3^k\bmod 59$ and $v_k=t_a\cdot y^{q_k}\in V$. Then \[Y(t_a)=\{v_0,\ldots,v_{28}\}\] consists exactly of the images of $t_a$ under the powers $y^{q_k}$. We follow the convention that subscripts of $v_i$ are read modulo $29$. The next lemma exhibits a relation between $Y(t_a)$ and $Y(t_b)$ for axes $t_a,t_b$ that form a good pair.

\begin{lemma}\label{lem_vs}
  Let $(t_a,t_b)$ be a good pair, say $t_a\cdot g=t_b$ and $y^g=y^{(3^k)}$ for some $g\in \M$. Let $Y(t_a)=\{v_0,\ldots,v_{28}\}$ and $Y(t_b)=\{v_0',\ldots,v_{28}'\}$. Then $v_i\cdot g=v_{i+k}'$, for every $i\in\{0,\ldots,28\}$, and $(w\cdot a)\cdot g=(w\cdot g)\cdot b$  for every $w\in V$.
\end{lemma}
\begin{proof}
  This follows from
\begin{align*} v_i\cdot g &= (t_a\cdot y^{q_i})\cdot g = (t_a\cdot g)\cdot y^{(3^{k+i})}= t_b\cdot y^{q_{i+k}}= v_{i+k}'
\end{align*}
Note that $t_a\cdot g=t_b$ if and only if $g^{-1}ag=b$, which implies the last claim.
\end{proof}
For every $r\in\{0,\ldots,28\}$ we now use the invariant form of the modulo-15 reduction of the 196,884-dimensional representation $\varrho$ to define three vectors $S_r,T_r,U_r$ of dimension $29$ with entries in $\mathbb{Z}_{15}$. The $j$-th components for $j\in\{0,\ldots,28\}$ of these vectors are defined by
\[S_r(j)=\langle t_a,v_{r+j}\rangle,\quad T_r(j)=\langle v_r\cdot a,v_{r+j}\rangle,\quad U_r(j)=T_{r+1}(j).\] This yields $29$ triples of vectors over $\mathbb{Z}_{15}$, namely, $(S_0,T_0,U_0)$, \ldots, $(S_{28},T_{28},U_{28})$.
\begin{definition}
The \emph{fingerprint} of $t_a$ is the lexicographically least element among the triples $(S_0,T_0,U_0)$, \ldots, $(S_{28},T_{28},U_{28})$.
\end{definition}

We follow the convention that the subscripts of $S,T,U$ are read modulo $29$. The next lemma implies that the two axes of a good pair have the same fingerprint.

\begin{lemma} Let $(t_a,t_b)$ be a good pair, say $t_a\cdot g=t_b$ and $y^g=y^{(3^k)}$ for some $g\in \M$. For every $r=0,\ldots,28$, if $S_r,T_r,U_r$ and $S_r',T_r',U_r'$ are the vectors defined with respect to $t_a$ and $t_b$, then \[S_r'=S_{r-k},\quad T_r'=T_{r-k},\quad\text{and}\quad U_r'=U_{r-k}.\]
\end{lemma}
\begin{proof}
  For every $j\in \{0,\ldots,28\}$, the invariance of the form and Lemma \ref{lem_vs} imply that
  \begin{align*}
    S_r'(j)&=\langle t_b,v'_{r+j}\rangle=\langle t_a\cdot g,v_{r+j-k}\cdot g\rangle=\langle t_a, v_{r+j-k}\rangle=S_{r-k}(j),\\
    T_r'(j)&= \langle v'_r\cdot b,v_{r+j}'\rangle=\langle (v_{r-k}\cdot a)\cdot g,v_{r+j-k}\cdot g\rangle=\langle v_{r-k}\cdot a,v_{r+j-k}\rangle=T_{r-k}(j),\\
  U_r'(j)&=T'_{r+1}(j)=T_{r+1-k}(j)=U_{r-k}(j).\qedhere
  \end{align*}
\end{proof}  

\begin{remark}\label{rem_firstFP}
When we have sampled $m$ random $2\myA$-axes $t_a$, we compute each $Y(t_a)$ and their fingerprint. Pairs of axes with the same fingerprint are \emph{candidates} for good pairs, and comparing the position of their fingerprint triples also determines the potential parameter $k$: if $t_a$ and $t_b$ have their fingerprint stored at position $r_a$ and $r_b$ respectively, then the (candidate) parameter $k$ is determined by $k=(r_b-r_a)\bmod 29$. The next step is to investigate whether the pair really is a good pair, that is, whether there exists $g\in \M$ with $t_a\cdot g=t_b$ and $y^g=y^{(3^k)}$. An approach for achieving this is outlined in the next section. If the approach fails, then we were not able to determine that $(t_a,t_b)$  is a good pair, and we continue our search.
\end{remark}

\subsection{A reduction to $K=2\cdot \B$}\label{sec_redK}
Suppose our previous search has produced a \emph{candidate} $(t_a,t_b)$ for a good pair, with parameter $k$. We aim to find a witness $g\in \M$ that proves that $(t_a,t_b)$ is a good pair, that is, an element $g\in \M$ such that $t_a\cdot g=t_b$ and $y^g=y^{(3^k)}$. 

Recall that $z\in 2\myA$ is our fixed $2\myA$-involution with $2\myA$-axis $t_z$; write $K=C_\M(t_z)=C_\M(z)=2\cdot \B$. Recall also that $t_a$ and $t_b$ have been constructed as
\[t_a=t_z\cdot r_a\quad\text{and}\quad t_b=t_z\cdot r_b\]
for some $r_a,r_b\in \M$. Define $A=y^{r_a^{-1}}$ and $B=(y^{(3^k)})^{r_b^{-1}}$. Finding a witness $g$ is now equivalent to finding $c\in K$ such that $A^c=B$, as shown by the following lemma.

\begin{lemma} With the previous notation, the map $h\mapsto r_ahr_b^{-1}$ is a bijection between the set $\{h\in \M: t_a\cdot h=t_b, y^h=y^{(3^k)}\}$ and the set $\{c\in K: A^c=B\}$.
\end{lemma}
\begin{proof}
  If $h\in \M$ satisfies $t_a\cdot h=t_b$ and $y^h=y^{(3^k)}$, then $c=r_ahr_b^{-1}$ satisfies 
  \[t_z\cdot c=t_a\cdot hr_b^{-1} =t_b\cdot r_b^{-1}=t_z,\]
  that is, $c\in K$; moreover, $A^c=y^{hr_b^{-1}}=(y^{(3^k)})^{r_b^{-1}}=B$.

  Conversely, if $c\in K$ satisfies $A^c=B$, then $h=r_a^{-1}cr_b$ has the property that
  \[y^h=y^{r_a^{-1}cr_b}=A^{cr_b}=B^{r_b}=y^{(3^k)}\]
  and $t_a\cdot h=t_z\cdot r_ah=t_z\cdot cr_b=t_z\cdot r_b=t_b$; recall that $c\in K$, so $t_z\cdot c=t_z$.
\end{proof}

The problem is now to decide whether  $A,B\in \M$ are $K$-conjugate: if $c\in K$ is determined with $A^c=B$, then $g=r_a^{-1}cr_b$ is the witness that proves that $(t_a,t_b)$ is a good pair.

\subsection{A reduction to $H=2^{1+23}.{\rm Co}_2$}\label{sec_redH}
We continue with the previous set-up. Recall that $K=2\cdot\B$, and that $\B$ has an involution centraliser isomorphic to $2^{1+22}.{\rm Co}_2$, see \cite[Section~6]{maxb}. In particular, there is a $2\myA$-involution $\tilde z\in K$ such that
$C_K(\tilde z)=2^{1+23}.{\rm Co}_2$; this can, for example, be routinely verified with the computer algebra system GAP \cite{ctbllib,gap} using the conjugacy class fusion and character tables data for $\M$ and $2\cdot \B$.  For the computation, we take $t_z$ to be the standard axis used by \texttt{mmgroup} and $t_{\tilde z}$ to be the fixed starting axis used by \texttt{RandomBabyAxis}. Thus
\[ H=C_K(\tilde z)=2^{1+23}.{\rm Co}_2\]
is represented directly as the stabiliser of $t_{\tilde z}$ in $K$; see Section \ref{sec_comp} for implementation details. We now describe another fingerprint and a reduction to a problem within the group $H$. Let \[\Omega=\{t_{\tilde z}\cdot k :k \in K\}\] be the $K$-orbit of the axis $t_{\tilde z}$; it has size  $|K|/|H|$.
 
\begin{definition}
  For a point $t_b\in \Omega$ with involution $b\in K$ and an element $C\in \M$ of order $59$ define $w_k=t_b\cdot C^{(3^k)}$ for $k=0,\ldots,28$, and define the fingerprint as the $87$-dimensional vector
  \[\Phi(t_b,C)=(\langle t_b,t_b\cdot C^j\rangle_{j=1,\ldots,29},
  \langle w_0\cdot b ,w_k\rangle_{k=0,\ldots,28},
  \langle w_1\cdot b ,w_{1+k}\rangle_{k=0,\ldots,28}).\]
\end{definition}

\begin{lemma}\label{lem_fp2}
If  $x\in K$, then $\Phi(t_b\cdot x,C^x)=\Phi(t_b,C)$. 
\end{lemma}
\begin{proof}
Note that for the vector $t_b\cdot x=t_{b^x}$ and element $C^x$, the associated element $w_k$ is defined as $t_b\cdot x\cdot (C^x)^{(3^k)}=t_b\cdot C^{(3^k)}x$. Moreover, the involution corresponding to $t_b\cdot x$ is $b^x$. Now the claim follows directly from  
\begin{align*}\langle t_b\cdot x,t_b\cdot x\cdot (C^x)^j\rangle&=\langle t_b\cdot x,t_b\cdot C^jx\rangle=\langle t_b,t_b\cdot  C^j\rangle,\\
  \langle t_b\cdot x\cdot C^x\cdot b^x, t_b\cdot x\cdot (C^x)^{(3^k)}\rangle&=\langle t_b\cdot C\cdot b, t_b\cdot C^{(3^k)}\rangle,\\
   \langle t_b\cdot x\cdot (C^x)^3\cdot b^x, t_b\cdot x\cdot (C^x)^{(3^{k+1})}\rangle&=\langle t_b\cdot C^3\cdot b, t_b\cdot C^{(3^{k+1})}\rangle.\qedhere
\end{align*}
\end{proof}

This fingerprint is useful for the following reason: 
the element $c\in K$ we seek satisfies $A^c=B$, and  Lemma \ref{lem_fp2} shows that $\Phi(t_b\cdot c,B)=\Phi(t_b,A)$ for every $t_b\in\Omega$. This shows that if  $t_{\tilde z}\cdot xc=t_{\tilde z}\cdot y$ for $x,y\in K$, then  necessarily  $\Phi(t_{\tilde z}\cdot y,B)=\Phi(t_{\tilde z}\cdot x,A)$. Thus, we now compute random elements $t_{\tilde z}\cdot x$ and $t_{\tilde z}\cdot y$ until $\Phi(t_{\tilde z}\cdot x,A)=\Phi(t_{\tilde z}\cdot y,B)$, and then investigate conditions under which there is $c\in K$ with $A^c=B$. The next lemma shows that this reduces to a problem in $H$.

\begin{lemma}For $x,y\in K$ define $A'=A^{x^{-1}}$ and $B'=B^{y^{-1}}$. If $c\in K$ satisfies $A^c=B$ and $t_{\tilde z}\cdot xc=t_{\tilde z}\cdot y$, then $h=xcy^{-1}\in H$ and $A'^h=B'$. Conversely, if $h\in H$ satisfies $A'^h=B'$, then $c=x^{-1}hy\in K$ satisfies $A^c=B$ and   $t_{\tilde z}\cdot xc=t_{\tilde z}\cdot y$.
\end{lemma}
\begin{proof}
 If $c\in K$ satisfies $A^c=B$ and $t_{\tilde z}\cdot xc=t_{\tilde z}\cdot y$, then $h=xcy^{-1}$ fixes $t_{\tilde z}$ and therefore  $h\in H$; moreover, $A'^h=B'$ by definition. Conversely, if $h\in H$ satisfies $A'^h=B'$, then $c=x^{-1}hy\in K$ satisfies $A^c=B$ and $t_{\tilde z}\cdot xc=t_{\tilde z}\cdot h y=t_{\tilde z}\cdot y$; note that $h\in H$ fixes $t_{\tilde z}$.
\end{proof}

The approach now is to construct $t_{\tilde z}\cdot x$ and $t_{\tilde z}\cdot y$  for random $x,y\in K$ until $\Phi(t_{\tilde z}\cdot x,A)=\Phi(t_{\tilde z}\cdot y,B)$, and then attempt to find $h\in H$ that conjugates $A^{x^{-1}}$ to $B^{y^{-1}}$. If successful, then $c=x^{-1}hy$ is an element in $K$ that conjugates $A$ to $B$, which in turn allows us to construct a witness, hence the generator $n$ we seek. 

\begin{remark}\label{rem_prob2}
  In practice, we construct  $t_{\tilde z}\cdot x_i$ and  $t_{\tilde z}\cdot y_j$ for random $x_1,\ldots,x_m,y_1,\ldots,y_m\in K$, and check whether $\Phi(t_{\tilde z}\cdot x_i,A)=\Phi(t_{\tilde z}\cdot y_j,B)$. By what is said above, this holds, for example, if $t_{\tilde z}\cdot y_j=t_{\tilde z}\cdot x_ic$ (for the unknown $c$ we seek), and so for a given $x_i$ and random $y_j$, the probability for this is at least  $1/|\Omega|$ for each pair. Thus, approximating with a Poisson distribution as before, the expected number of pairs with the same fingerprint is at least $\lambda_m=m^2/|\Omega|$, and the probability that we have at least one match is approximately $1-e^{-\lambda_m}$. For $m=6{,}000{,}000$, the probability of finding a match is $>0.95$. As before, we use this as a heuristic argument, rather than as a precise estimate of the probability.
\end{remark}

\subsection{A reduction to $Q=O_2(H)=2^{1+23}$}\label{sec_redQ}
We continue with the previous set-up: We suppose we have constructed elements $A'$ and $B'$ and seek $h\in H$ such that $A'^h=B'$. To simplify the exposition, we write $A=A'$ and $B=B'$.

In our computation, $H$ is contained in the centraliser of a $2\BB$-involution, that is, in a maximal subgroup $G=2_+^{1+24}.{\rm Co}_1$ of $\M$. This is relevant because \texttt{mmgroup} provides specialised methods for computing within $G$. Write $E=O_2(G)$ and $Z(G)=\langle\zeta\rangle\leq E$; it is well-known that $E/Z(G)\cong \Lambda/2\Lambda\cong\mathbb{Z}_2^{24}$, where $\Lambda$ denotes the Leech lattice. The conjugation action of $H$ on this module $W=E/Z(G)$ has kernel $Q=O_2(H)=2^{1+23}$; if $h\in H$ and $u\in W$, then we write $u\cdot h$ for this action. We choose points $u_1,u_2,u_3\in W$ and  build the corresponding stabiliser chain, that is, a subgroup chain
\[H=H_0>H_1>H_2>H_3=Q\]
where each $H_i={\rm Stab}_{H_{i-1}}(u_i)$, together with transversal elements: for every $i$ and every $w$ in the orbit $u_i\cdot H_{i-1}$ we store $\sigma_w\in H_{i-1}$ such that $u_i\cdot \sigma_w=w$. If $u\in W$, then the preimages in $E$ are $\{\hat{u},\hat{u}\zeta\}$, where $\hat u\in E$ is any lift of $u$. For an element $C\in \M$ of order $59$ we define the following fingerprint, comprising an unordered pair of ordered tuples of element orders:
\[\Psi(u,C)=\left\{\bigl(|\hat u C^{(3^i)}|\bigr)_{i=0}^{28},\bigl(|\hat u\zeta C^{(3^i)}|\bigr)_{i=0}^{28}\right\}.\]
As before, the fingerprint allows us to formulate a necessary condition for the element we seek.

\begin{lemma}\label{lem_fppsi}
If $h\in H$ satisfies $A^h=B$, then $\Psi(u,A)=\Psi(u\cdot h,B)$ for all $u\in W$.
\end{lemma}
\begin{proof}
By construction, if $\hat{u}$ is a lift of $u$, then $\hat{u}^h$ is a lift of $u\cdot h$. Let $v\in\{1,\zeta\}$ and $e=3^i$ for some $i$; now the claim follows from $|\hat{u}vA^e| = |(\hat{u}vA^e)^h|=|\hat{u}^h v B^e|$.
\end{proof}

We now work along the stabiliser chain and start sampling points $u,w$ in the orbit $u_1\cdot H_0$ until we have $\Psi(u,A)=\Psi(w,B)$; this holds, for example, when $w=u\cdot h$ where $h$ is the element we aim to construct. We then write $u=u_1\cdot\sigma_u$ and $w=u_1\cdot\sigma_w$. If indeed $u\cdot h=w$, then $h_1=\sigma_uh\sigma_w^{-1}$ fixes $u_1$ and therefore lies in $H_1$; moreover, $h_1$ conjugates $A_1=A^{\sigma_u^{-1}}$ to $B_1=B^{\sigma_w^{-1}}$. This has reduced our original problem of finding $h\in H_0$ with $A^h=B$ to finding $h_1\in H_1$ with $A_1^{h_1}=B_1$. We iterate this process until we have obtained elements $A_3$ and $B_3$, and the task is to find $h_3\in H_3=Q$ such that $A_3^{h_3}=B_3$. This has reduced our problem to a conjugacy problem involving the subgroup $Q$. Since the group $Q$ has size $2^{24}$,  a brute-force search suffices (even though in practice this can be made slightly more efficient).

\section{On the computation}\label{sec_comp}
Our computation used \texttt{mmgroup}~1.0.8 on the CREATE cluster \cite{create}. Since fingerprint computations dominated the search, this stage was implemented in multithreaded C++, calling the low-level \texttt{mmgroup}
routines directly; every reported match was subsequently reconstructed in Python using \texttt{mmgroup}. 
 The search for our first fingerprint match generated $286{,}000{,}000$ axes and produced exactly one match, see Remarks \ref{rem_prob1} and \ref{rem_firstFP}. Remarkably, this match was sufficient to eventually construct the generators in Theorem \ref{thm:mainthm}. The next step of our algorithm constructed two further lists, each containing about $6{,}000{,}000$ axes, see Remark \ref{rem_prob2}; comparing these two lists produced three matches, which were then passed through the stabiliser-chain process described in Section \ref{sec_redQ}. The first of the resulting recovery computations to finish provided our generators; the other two were then stopped.

We note that no generic centraliser and stabiliser-chain calculations in the Monster were
required. The group $K=2\cdot\B$ was represented implicitly as the stabiliser
of the standard axis, and the \texttt{mmgroup} class
\texttt{RandomBabyAxis} supplied sampled points in the $K$-orbit of
$t_{\tilde z}$ together with corresponding transversal elements in $K$. The group
$H=C_K(\tilde z)$ was therefore represented as the stabiliser of this fixed
axis in \emph{standard coordinates}: eight fixed elements of $H$ generated its
${\rm Co}_2$ image on $W=O_2(G)/Z(G)$, while $Q=O_2(H)$ was represented by an explicit
24-generator basis. (Recall that $G=2_+^{1+24}.{\rm Co}_1$ is the centraliser of the standard $2\BB$-element in $\M$.) The orbits and transversal elements required for the
stabiliser chain described in Section~\ref{sec_redQ} were computed on $W$
using the low-level class \texttt{mmgroup.general.Orbit\_Lin2}.

For the proof of Corollary~\ref{cor:no-psl}, we found the required elements $t$ and $w$ by first working in the subgroup $3.{\rm Fi}_{24}'.2$ of $\M$ given in the database accompanying \cite{dlpp},
and then applying a fingerprinting-and-recovery procedure to obtain conjugates in $N_{\M} (n)$.

We conclude with two technical remarks, providing some details on the computation resources and describing the assistance of AI systems in the implementation phase.

\begin{remark}
During the two bulk axis-generation stages, at most five CREATE compute nodes
were used simultaneously. Each was allocated 124 CPU slots by the Slurm
scheduler and $32\,\mathrm{GB}$ of memory, giving a peak allocation of 620 CPU
slots. These figures describe scheduler allocations rather than measured
processor utilisation. The global search for the first fingerprint match,
comprising $286{,}000{,}000$ axes, took about 34 hours and used $19{,}069$
allocated CPU-hours, including reruns. The second-stage search took about
3.5 hours, and the successful recovery job took about 2.5 hours using one CPU
slot. In total, approximately 43 hours elapsed from launching the full
first-stage search to obtaining the generators.
\end{remark}

\enlargethispage{1cm}
\begin{remark}
Our computational search was carried out largely autonomously by AI systems:
GPT-5.6 Sol Pro produced the initial design blueprint, which OpenAI Codex then
audited against \texttt{mmgroup}; this identified and corrected
substantive problems, and eventually led to the final
computational pipeline. Codex wrote the required C++, Python, testing,
certification, and Slurm code; it also orchestrated the CREATE computation,
including job submission, monitoring, failure recovery, restarts, candidate
processing, and final certification. The correctness of our main results and the mathematics described in this paper do not depend on trusting AI-generated reasoning.
\end{remark}

\newpage 
\begin{small}

\end{small}

\end{document}